\documentclass[12pt]{amsart}

\usepackage{hyperref}
\usepackage{amssymb}

\usepackage{graphicx}
\usepackage{ifthen}
\usepackage{pict2e}
\usepackage{xargs}
\usepackage{xspace}
\usepackage{xcolor}
\usepackage{pgf,tikz}
\usepackage{pgfplots}
\usepackage{environ}
\usepackage{caption}
\usetikzlibrary{arrows}
\usepackage{enumitem}
\usepackage{xifthen}
\newtheorem{theorem}{Theorem}[section]

\newtheorem{proposition}[theorem]{Proposition}

\theoremstyle{definition}\newtheorem{definition}[theorem]{Definition}

\theoremstyle{definition}
\theoremstyle{definition}

\numberwithin{equation}{section}

\newcommand{\bd}{\begin{definition}}
	\newcommand{\ed}{\end{definition}}

\DeclareMathOperator{\ran}{ran}

\newcommand{\enum}{Enum}

\newcommand{\Infinitesubsets}[1]{[#1]^\omega}

\pgfplotsset{soldot/.style={color=blue,only marks,mark=*}}
\usetikzlibrary{math}
\usetikzlibrary{arrows.meta}
\usetikzlibrary{shapes,decorations}

\usetikzlibrary{decorations.pathmorphing}
\usepackage{pgfplots}

\title{Tall $F_\sigma$ subideals of tall analytic ideals}

\author{Jan Greb\' ik}
\address{
Mathematics Institute.
University of Warwick,
Coventry CV4~7AL, UK}
\email{jan.grebik@warwick.ac.uk}
\author{Zolt\'an Vidny\'anszky}
\address{Kurt G\"odel Research Center for Mathematical Logic, Universit\"at Wien, W\"ah\-rin\-ger Strasse 25, 1090 Wien, Austria}
\email{zoltan.vidnyanszky@univie.ac.at}

\thanks{The first author was supported by Leverhulme Research Project Grant RPG-2018-424.
The second author was supported by the National Research, Development and Innovation Office -- NKFIH, grant no.~113047 and FWF Grants P29999 and M2779.}

\begin{document}
	\maketitle
	\begin{abstract}
		Answering a question of Hru\v s\' ak, we show that every analytic tall ideal on $\omega$ contains an $F_\sigma$ tall ideal. We also give an example of an $F_\sigma$ tall ideal without a Borel selector.
		
	\end{abstract}
 Recall that an (non-trivial) \emph{ideal} $\mathcal{I}$ on $\omega$ is a collection of subsets of $\omega$ with the following properties for every $A, B \subset \omega$:
 \begin{enumerate}
 	\item if $A, B \in \mathcal{I}$ then $A \cup B \in \mathcal{I}$,
 	\item if $A \in \mathcal{I}$ and $B \subset A$ then $B \in \mathcal{I}$,
 	\item if $A$ is finite, then $A \in \mathcal{I}$,
 	\item $\omega \not \in \mathcal{I}$.
 \end{enumerate}

  Identifying subsets of $\omega$ with their characteristic functions, an ideal can be considered to be a subset of $2^\omega$, and as such, can have definability properties, i.e., being Borel, analytic etc.
  Recall that a set $S \subset 2^\omega$ is analytic if it is a continuous image of the Baire space.
  
  Analytic ideals are remarkably well-behaving, and play a central role in the study of ideals on $\omega$ (see, e.g., \cite{uzcategui2019ideals}, \cite{soleckianalytic}, \cite{todorcevicgaps}). 
  
  %The most natural ideals have a further property, called tallness.
  One of the properties of ideals that is enjoyed by most natural examples is called tallness.
  A family $H$ of subsets of $\omega$ is \emph{tall} if for every $x \in [\omega]^\omega$ there exists a $y \in [x]^\omega$ with $y \in H$.
  Here $\Infinitesubsets{S}$ $([S]^{<\omega})$ denotes the collection of countably infinite (finite) subsets of a set $S$.
  This notion is fundamental, for example, in the study of \emph{Kat\v etov order} on ideals, see~\cite{hrusak}.
  Recall that an ideal $\mathcal{I}$ is \emph{Kat\v etov below} an ideal $\mathcal{J}$ if there is a function $f:\omega\to \omega$ such that $f^{-1}(A)\in \mathcal{J}$ whenever $A\in \mathcal{I}$.

  It has been asked by Hru\v s\' ak \cite[Question~5.6]{hrusak}, whether every analytic tall ideal contains an $F_\sigma$ tall ideal.
  It is easy to see that this question is equivalent to the seemingly weaker problem whether tall $F_\sigma$ ideals are dense among tall analytic ideals in the Kat\v etov order. 
  We give a simple, affirmative answer to these questions.

	\begin{theorem}
	    \label{t:main}
		Let $\mathcal{I}$ be an analytic tall ideal on $\omega$. Then there exists an $F_\sigma$ tall ideal $\mathcal{J} \subset \mathcal{I}$.
		In particular, tall $F_\sigma$ ideals are dense among tall analytic ideals in the Kat\v etov order.
	\end{theorem}
\begin{proof}
	First we find a closed tall set $H \subset \mathcal{I}$.
	
	Let us denote by $INC$ the collection of the strictly increasing $\omega \to  \omega$ functions, note that this is a Polish space with the topology inherited from $\omega^\omega$. For $x \in [\omega]^\omega$ we will denote by $\enum(x)$ the increasing enumeration of the elements of $x$.
	
	Let
	$$\mathcal{I}'=\{\enum(x): x\in  \mathcal{I}\cap [\omega]^\omega\}.$$
	Clearly, $\mathcal{I}'$ is analytic. 
    By standard descriptive set theoretic facts (see, \cite[Exercise~14.3]{kechrisbook}) we can pick a closed set $F \subset  INC \times INC$ such that $\mathcal{I}'$ is the projection of $F$ to the first coordinate. Let us define a map $\Psi:F \to 2^\omega$ by letting 
	\[\Psi(f,g)(n)=1 \iff n \in \ran(f \circ g),\] for every $n 
	\in \omega$. Clearly, $\Psi$ is continuous and $\Psi(f,g)\in [\omega]^\omega$ for every $(f,g)\in F$.
	
	We claim that the set $H=\overline{\Psi(F)}$ works. First note that $H$ is tall: indeed, if $x\in \Infinitesubsets{\omega}$ then there exists $y \in \Infinitesubsets{x}$ with $y \in \mathcal{I}$, in turn there exists some $g \in INC$ such that $(\enum(y),g) \in F$. So, $\Psi(\enum(y),g) \subset \ran(\enum(y))=y \subset x$.
	
	Second, we show that $H \subset \mathcal{I}$. Towards a contradiction, assume that $x \in H \setminus \mathcal{I}$. Then, by $[\omega]^{<\omega} \subset \mathcal{I}$ we have that $x$ is infinite. Pick a sequence $(f_n,g_n)_{n \in \omega}$ of elements of $F$ such that $\Psi(f_n,g_n) \to x$. Let $h=\enum(x)$. Then, for any $k \in \omega$ there exists an $n_k$, such that $(f_{m} \circ g_{m})(k)=h(k)$ holds for $m \geq n_k$. In particular, using monotonicity, we get that $g_m(k) \leq h(k)$ and $f_m(k) \leq (f_m\circ g_m)(k) \leq h(k)$, for every $m \geq n_k$. But then there is $(f,g)\in INC\times INC$ and a subsequence such that $(f_{n_l},g_{n_l})_{l \in \omega} \to (f,g)$. By the continuity of $\Psi$, and $F$ being closed we have $\Psi(f,g)=x$ and, as $f \in \mathcal{I}'$, $x \subset \ran(f) \in \mathcal{I}$, a contradiction.

	In order to finish the proof of the theorem we just need the following easy observation:
	\[\text{If $H \subset 2^\omega$ is $\sigma$-compact, then so is the ideal generated by $H$.} \tag{*}\]
	Indeed, if $S \subset 2^\omega$ is compact, then so are the sets $\{x: (\exists y \in S)(x \subset y)\}$ and $\{x \cup y: x,y \in S\}$ (the latter one is by the continuity of the $\cup$ operation). Thus, the analogous statement holds for $\sigma$-compact sets, as well.
	
	So, the ideal generated by $H\cup [\omega]^{<\omega}$ is $\sigma$-compact, tall, and contained in $\mathcal{I}$.
	
\end{proof}
    Let us point out that the idea behind the above argument has already appeared in \cite{solovay} and \cite{toden}.

    In what follows, we construct two $F_\sigma$ tall ideals without Borel selectors (recall that a \emph{selector} for a tall ideal on $\omega$ is a map $\phi:[\omega]^\omega \to \mathcal{I}$ such that for each $S \in [\omega]^\omega$ we have $\phi(S) \in [S]^\omega\cap \mathcal{I}$). 
    The existence of such an object has been proved by Uzcategui and the first author \cite{grebik}, without giving an explicit example. 
    
    Our first example is based on Theorem \ref{t:main} and will use effective descriptive set theory, while the second example will be completely elementary. An argument for including the first one is that it uses a peculiar analytic $\sigma$-ideal, which seems to be new. 
    
    We will define an ideal on $2^{<\omega}$ (a countable set). Note that $2^{<\omega}$ is endowed with a natural tree structure, and the branches of the tree determine a continuum size family of pairwise almost disjoint subsets of $2^{<\omega}$. 
    
    For $f \in 2^\omega$ let us denote by $f'$ the set $\{f\restriction n: n \in \omega\}\subseteq 2^{<\omega}$. In the definition below we will use notions of effective descriptive set theory (see, e.g., \cite{moschovakis2009descriptive}). In order to do that we have to fix a recursive bijection identifying $\omega$ with $2^{<\omega}$, and hence endowing $\mathcal{P}(2^{<\omega})$ with a recursive Polish structure.
    
    Let $S \in A_0$ iff $S \subset  2^{<\omega}$ and $S$ is an antichain, and let $S \in A_1$ iff 
    \[\exists f \in 2^\omega \ (S \subset f' \land \forall T \in \Delta^1_1(f') (|S \cap T|<\infty \lor |S \setminus T|<\infty)).\]
    
    Roughly speaking, we collect those subsets of a given $f'$ which cannot be split into two infinite pieces by a real computable from $f'$. 
    
    By the Spector-Gandy theorem, the set $A_1$ is analytic, and, clearly, the set $A_0$ is closed. It is not hard to check that the ideal $\mathcal{I}$ generated by $A_0 \cup A_1$ is analytic, tall  and admits no Borel selector. Now, an application of Theorem \ref{t:main} to $\mathcal{I}$ yields an $F_\sigma$ tall ideal without a Borel selector.
    
    Finally, let us construct the second, more concrete example. The ideal is defined on $2^{<\omega}$.
    Again, let $S \in A_0$ iff $S\subset 2^{<\omega}$ is an antichain. Moreover, fix a closed set $F \subset 2^\omega \times \omega^\omega$ that projects onto $2^\omega$ and admits no Borel unformization (see, \cite[Section 18]{kechrisbook} or \cite[4D.11]{moschovakis2009descriptive}). Let $S \in B$ if and only if
    \[\exists f \in 2^\omega \ (S \subset f' \land \exists g \in \omega^\omega\ ((f,g) \in F \land g \leq \enum(S))),\]
     where we abuse the notation and write $\enum(S)$ for the enumeration of indices $n\in \mathbb{N}$ such that $f\upharpoonright n\in S$, and $\leq$ stands for the pointwise ordering of functions from $\omega^\omega$. 
     
     Finally, let $\mathcal{J}$ be the ideal generated by $A_0 \cup B$. 
     
     \begin{proposition}
     $\mathcal{J}$ is an $F_\sigma$ tall ideal that doesn't admit a Borel selector.
     \end{proposition}
     
     \begin{proof}
         We first check that the set $A_0 \cup B \cup [2^{<\omega}]^{<\omega}$ is $\sigma$-compact in $\mathcal{P}(2^{<\omega})$.
         This is sufficient to show that $\mathcal{J}$ is $F_\sigma$ by $(*)$ above. The argument will be rather similar to the proof of Theorem \ref{t:main}.
         
         Since $A_0$ is closed, it is enough to show that $\overline{B}\subseteq B\cup [2^{<\omega}]^{<\omega}$.
         Let $x_n\to x$ be a convergent sequence where $x_n\in B$, and suppose that $x\not\in [2^{<\omega}]^{<\omega}$. 
         Pick witnesses $(f_n,g_n)\in 2^\omega\times \omega^\omega$ to $x_n$, that is, $(f_n,g_n) \in F$ and $g_n \leq \enum(x_n)$.
         It is easy to see that there is an $f\in 2^\omega$ such that $f_n\to f$ and $x\subseteq f'$.
         Similarly, one gets $\enum(x_n)\to \enum(x)$.
         Hence, there must be $g\in \omega^\omega$ and a subsequence $g_{n_l}\to g$ in $\omega^\omega$ becuase $g_{n_l}\le \enum(x_n)$.
         So, we have $(f_{n_l},g_{n_l})\to (f,g)$ and $g\le \enum(x)$.
         Consequently, $x\in B$ because $F$ is closed.

         %It is enough to show that $\overline{A_0 \cup B} \subset A_0 \cup B \cup [2^{<\omega}]^{<\omega}$. So let $x_n \to x$ be a convergent sequence from $A_0 \cup B$. By its definition, the set $A_0$ is closed, so we can assume that $x_n \in B$ for each $n \in \omega$. Pick witnesses $(f_n,g_n)$ to $x_n \in B$, that is, $(f_n,g_n) \in F$ and $g_n \leq \enum(x_n)$. If $x \not \in [2^{<\omega}]^{<\omega}$, then there must be a convergent subsequence $(f_{n_k},g_{n_k}) \to (f,g)$, with $g \leq \enum(x)$. But then, as $F$ is closed we have $(f,g) \in F$, $x \subset f'$ and thus $x \in B$.
         
         To see that $\mathcal{J}$ is tall it is enough to realize that every $x\in [2^{<\omega}]^{\omega}$ either satisfies $|f'\cap x|=\omega$ for some $f\in 2^\omega$ or can be covered by finitely many antichains by K\" onig's theorem.
         In the former case one can easily find an infinte subset of $x$ that dominates some $g$ where $(f,g)\in F$.
         
         Suppose for a contradiction that $\mathcal{J}$ admits a Borel selector $S$.
         Let $K=A_0\cup\overline{B}$.
         It follows from the second paragraph above that $K$ is closed, $K\subseteq \mathcal{J}$, $K$ generates $\mathcal{J}$ and $B=\overline{B}\cap [2^{<\omega}]^{\omega}$.
         By \cite[Proposition~4.5]{grebik} we may assume that $S(x)\in K$ for every $x\in [2^{<\omega}]^\omega$.
         In particular, $S(f')\in B$ for every $f\in 2^\omega$.
         Define 
         $$G=\{(f,g)\in F:g\le \enum(S(f'))\}.$$
         Then $G\subseteq F$ is a Borel set that has $\sigma$-compact vertical sections and projects to $2^\omega$ by the definition of $B$.
         Such a set has a Borel uniformization by \cite[Theorem 18.18]{kechrisbook}, contradicting the choice of $F$.

         %For the nonexistence of a Borel selector notice that by the fact that the sets in the family $\{f':f \in 2^\omega\}$ are almost disjoint, for every $f \in 2^\omega$ the intersection $\mathcal{J} \cap [f']^\omega$ equals to the elements of $B \cap [f']^\omega$ modulo finite modification. In particular, the existence of a Borel selector for $\mathcal{J}$ would imply the existence of a Borel map $s:2^\omega \to [\omega]^\omega$ such that for each $f \in 2^\omega$ we had that $s(f) \subset f'$ and for some $g \in \omega^\omega$ with $(f,g) \in F$ and $M \in \omega$ we had $\forall n \in N \ (g(n) \leq \enum(s(f))(n+M))$. But then the set $\{(f,g) \in F: \exists M \in \omega \ \forall n \in \omega \ (g(n) \leq \enum(s(f))(n+M))\}$ would be a Borel subset of $F$, with full first projection and $\sigma$-compact vertical sections. Such a set has a Borel uniformization by \cite[Theorem 18.18]{kechrisbook}, contradicting the choice of $F$.
     \end{proof}

	\bibliographystyle{abbrv}
	\bibliography{bblhrusak}
	
\end{document}